\documentclass{amsart}

\usepackage{amssymb}

\usepackage{amsmath, amscd, mathrsfs, yfonts, bbm}

\usepackage{graphicx}

\makeatletter
\@namedef{subjclassname@2010}{%
  \textup{2010} Mathematics Subject Classification}
\makeatother

\newtheorem{thm}{Theorem}[section]

\newtheorem{lem}[thm]{Lemma}
\newtheorem{prop}[thm]{Proposition}

\theoremstyle{definition}
\newtheorem{defin}[thm]{Definition}

\numberwithin{equation}{section}

\newcommand{\su}{\succeq}
\newcommand{\kr}{\text{Kr}}

\newcommand{\cR}{{}^cR}
\newcommand{\cT}{{}^cT}

\newcommand{\cE}{\mathcal{E}}
\newcommand{\CE}{\widetilde{\mathcal{E} } }

\newcommand{\mK}{\mathcal{K}}

\newcommand{\T}{{}^cT}

\newcommand{\lra}{\longrightarrow}

\newcommand{\ck}{{}^c\kappa}
\newcommand{\ct}{{}^ct}

\newcommand{\NC}{\textrm{NC}}
\newcommand{\NCL}{\textrm{NCL}}

\newcommand{\com}{\widetilde{\omega}}

\begin{document}




\title[Multiplication of operator-valued c-free random variables]{On the multiplication of operator-valued c-free random variables}

\author[M. Popa]{Mihai Popa}
\address{Department of Mathematics, University of Texas at San Antonio, One UTSA Circle
San Antonio, Texas 78249, USA}
\email{mihai.popa@utsa.edu}

\author[V. Vinnikov]{Victor Vinnikov}
\address{Department of Mathematics, Ben-Gurion University of the Negev, Be'er Sheva 8410501, 
Israel}
\email{vinnikov@cs.bgu.ac.il}

\author[J.-C. Wang]{Jiun-Chiau Wang}
\address{Department of Mathematics and Statistics, University of Saskatchewan,
Saskatoon, Saskatchewan S7N 5E6, Canada}
\email{jcwang@math.usask.ca}

\date{}

\begin{abstract}
The paper discusses some results concerning the multiplication of non-commutative random variables that are c-free with respect to a pair $( \Phi, \varphi) $, where $ \Phi $ is a linear map with values in some Banach or C$^\ast$-algebra and  $ \varphi $ is scalar-valued. In particular, we construct a suitable analogue of the Voiculescu's $ S $-transform for this framework.
\end{abstract}



\maketitle

\section{Introduction}
 The terminology ``c-free independence''(or c-freeness) was first used in the 1990's by M. Bozejko, R. Speicher and M.Leinert (see \cite{bs}, \cite{bls}) to denote a relation similar to D.-V. Voiculescu's free independence, but in the framework of algebras endowed with two linear functionals (see Definition \ref{def:1} below). The additive c-free convolution and the analytic characterization of the correspondent infinite divisibility was described in 1996 (see \cite{bls}); appropriate instruments for dealing
 with the multiplicative c-free convolution appeared a decade later, in \cite{mvp-jcw}. There, for $ X $ a non-commutative random variable, we define an analytic function $\cT_X(z) $, inspired by Voiculescu's $ S $-transform, such that if $ X $ and $ Y $ are c-free , then $\cT_{XY}(z)= \cT_X(z) \cdot \cT_Y(z) $. Alternate proofs of this result were given in \cite{mp} and \cite{mp-proc}.
 
  The present work discusses the multiplicative c-free convolution in the framework of \cite{mlot}, namely when one of the functionals is replaced by a linear map with values in a (not necessarily commutative) Banach algebra. In particular, we show that the combinatorial methods from \cite{mp} can be adapted to this more general framework. Notably, in the case of 
 free independence over some Banach algebra, as showed in \cite{dykema}, \cite{dykema2}, the analogue of Voculescu's $S$-transform satisfies a ``twisted multiplicative relation'', namely $ T_{XY}(b)= T_X( T_Y(b)\cdot b \cdot T_Y(b)^{ - 1})\cdot T_Y(b) $. The main result of the present work, Theorem \ref{mainthm}, shows that the non-commutative $\cT$-transform satisfies the usual multiplicative relation: $ \cT_{XY}(z) = \cT_X(z) \cdot \cT_Y(z) $ for $ X, Y $ c-free non-commutative random variables.
 
The paper is organized as follows. Section 2 presents some preliminary notions and results, mainly concerning the lattice of non-crossing linked partitions and its connection to free and c-free cumulants and to $t$- and ${}^ct$-coefficients. The main result of the section is Proposition \ref{vanishtcoeff}, the characterization of c-freeness in terms of ${}^ct$-coefficients. Section 3 restates the results on planar trees used in \cite{mp} and utilize them for the main result of the paper, Theorem \ref{mainthm}. Section 4 discusses some aspects of infinite divisibility for the multiplicative c-free convolution in operator-valued framework. The results from the scalar case (see \cite{mvp-jcw}) can be easily extended to the framework of a commutative algebra of operators, but they are generally not valid in the non-commutative case.

\section{Framework and notations}
  
\subsection{Non-crossing partitions and c-free cumulants} In the first part of this paper we will consider $ A $ and $ B $ to be two unital Banach algebras and
$ \varphi:A \lra \mathbb{C} $, respectively $ \Phi: A \lra B $ to be two unital and linear maps. If $ A $ and $ B $ are C$^{\ast}$- or von Neumann algebras, then we will require that $ \varphi $, respectively $ \Phi $ to be positive, respectively completely positive.

\begin{defin}\label{def:1}

$ A_1 $ and $ A_2 $, two unital subalgebras of $ A $, are said to be \emph{c-free} with respect to
 $ ( \Phi, \varphi )$ if for all $ n $ and all  $ a_1, a_2, \dots, a_n $ such that 
$\varphi(a_j)=0$ and 
  $a_j\in A_{\varepsilon(j)}$ with 
  $\varepsilon(j)\in\{1,2\}$
  and 
  $\varepsilon(j )\neq\varepsilon(j + 1 ) $
 we have that
\begin{enumerate}
\item[(i)] $ \varphi ( a_1 \cdots a_n ) = 0 $
\item[(ii)] $ \Phi ( a_1 \cdots a_n )= \Phi(a_1 ) \cdots \Phi ( a_ n ) $.
\end{enumerate}
Two elements,  $ X $ and $ Y, $ of $ A $ are said two be \emph{c-free} with respect to $ ( \Phi, \varphi ) $ if the unital subalgebras of $  A $  generated by $ X $ and $ Y $ are c-free, as above. If $ A $ is a C$^\ast$- or a von Neumann algebra, then we will require that the unital C$^\ast$-, respectively von Neumann subalgebras of $ A $ generated by $ X $ and $ Y $ are c-free. If only condition {(i)} holds true, then the subalgebras $A_1, A_2 $ (respectively the elements $ X, Y $) are said to be \emph{free} with respect to $\varphi$.
\end{defin}

As shown in \cite{bls}, \cite{mp-commstoch}, there is a convenient combinatorial characterization of c-freeness in terms of non-crossing partitions, that we will summarize below.

  A non-crossing partition $\gamma$ of the ordered set
 $\{1,2,\dots,n\}$ is a collection $C_1,\dots, C_k$
 of mutually disjoint subsets of $\{1,2,\dots,n\}$, called blocks, such that their union is the entire set $ \{ 1, 2, \dots, n \} $ and there are no crossings, in the sense that 
  there
 are no two blocks $C_l, C_s$ and $i<k<p<q$ such that $i,p\in C_l$
 and $k,q\in C_s$.

\textbf{Example 1}: Below is represented graphically the
non-crossing
partition\\
 $\pi=(1,5,6), (2,3), (4),(7,10),(8,9)$:

 \setlength{\unitlength}{.14cm}
 \begin{equation*}
 \begin{picture}(10,8)

\put(-16.5,0.5){1}

\put(-15.9,3){\circle*{1}} \put(-16,3){\line(0,1){5}}

\put(-12.5,0.5){2}

\put(-12,3){\circle*{1}} \put(-12,3){\line(0,1){3}}

\put(-8.5,0.5){3}

\put(-8,3){\circle*{1}}\put(-8,3){\line(0,1){3}}

\put(-12,6){\line(1,0){4}}

\put(-4.5,0.5){4}

\put(-3.9,3){\circle*{1}}\put(-4,3){\line(0,1){3}}

\put(-0.5,0.5){5}

\put(0,3){\circle*{1}}\put(0,3){\line(0,1){5}}

\put(3.5,0.5){6}

\put(4,3){\circle*{1}}\put(4,3){\line(0,1){5}}

\put(-16,8){\line(1,0){20}}

\put(7.5,0.5){7}

\put(8.1,3){\circle*{1}}\put(8,3){\line(0,1){5}}

\put(11.5,0.5){8}

\put(12.1,3){\circle*{1}}\put(12,3){\line(0,1){3}}

\put(15.5,0.5){9}

\put(16.2,3){\circle*{1}}\put(16,3){\line(0,1){3}}

\put(12,6){\line(1,0){4}}

\put(19,0.5){10}

\put(20.2,3){\circle*{1}}\put(20,3){\line(0,1){5}}

\put(8,8){\line(1,0){12}}

 \end{picture}
  \end{equation*}

 The set of all
non-crossing partitions on the set  $ \{ 1, 2, \dots, n \} $ will be denoted by
$NC(n)$.   It    has a lattice structure with respect to the reversed
refinement order, with the biggest, respectively smallest element
$\mathbbm{1}_{n}=(1,2,\dots,n)$, respectively $0_{n}=(1),\dots,(n)$.
For $\pi,\sigma\in NC(n)$ we will denote by $\pi\bigvee\sigma$ their
join (smallest common upper bound).

For $\gamma\in NC(n)$, a block $B=(i_{1},\dots,i_{k})$ of $\gamma$
will be called \emph{interior} if there exists another block
$D\in\gamma$ and $i,j\in D$ such that $i<i_{1},i_{2},\dots,i_{k}<j$.
A block will be called \emph{exterior} if is not interior. The set
of all interior, respectively exterior blocks of $\gamma$ will be
denoted by $\text{Int}(\gamma)$, respectively $\text{Ext}(\gamma)$. The set $ \text{Ext} ( \gamma) $ is totally ordered by the value of the first element in each block. 

For $ X_1, \dots, X_n \in A $, we define the free, respectively c-free, cumulants 
$ \kappa_n (X_1, \dots, X_n ) $, respectively $ {}^c\kappa_n ( X_1, \dots, X_n )  $ as the multilinear maps from $ A^n $ to $ \mathbb{C} $, respectively $ B $,  given by the recurrences below:
\begin{align*}
\varphi( X_1 \cdots X_n ) &
= 
\sum_{ \gamma \in \NC( n ) } 
\prod_{\substack{ C = \text{block in} \gamma \\ C = (i_1, \dots, i_l ) } }  \kappa_l (X_{ i_1}, \dots, X_{ i_l } )\\
\Phi ( X_1 \dots X_n ) & 
= 
 \sum_{ \gamma \in \NC( n ) } 
 [ 
\prod_{\substack{ B \in \text{Ext} (  \gamma ) \\  B  = ( j_1, \dots, j_l ) } } \ck_l (  X_{ j_1 } \cdots X_{ j_l } ) ]
 \cdot 
[ 
\prod_{\substack{ D  \in \text{Int} ( \gamma ) \\  D = (i_1, \dots, i_s ) } }  \kappa_s (X_{ i_1}, \dots, X_{ i_s } )
]
\end{align*}
with the convention that if $ \Lambda = \{ \alpha(1),\alpha(2), \dots, \alpha(n) \} $ is a totally ordered set and $ \{ X_{\lambda} \}_{\lambda \in \Lambda} $ is a collection of elements from $ B $, then 
\[
 \prod_{ \lambda \in \Lambda } X_\lambda = X_{ \alpha(1) } \cdot X_{ \alpha(2)}\cdots X_{ \alpha(n) }. 
 \]

We will us the shorthand notations $ \kappa_n ( X ) $  for $ \kappa_n ( X, \dots, X ) $ and 
$ \ck_n (X) $ for $ \ck_n ( X, \dots, X ) $.

As shown in \cite{ns}, \cite{vdn}, and in \cite{mp-commstoch}, if $ X_1 $ and $ X_2 $ are c-free, then 
\begin{align}
R_{ X_1 + X_2 }(z) & = 
R_{X_1} ( z ) + R_{X_2} ( z ) \\
\cR_{ X_1 + X_2} ( z ) & = 
\cR_{X_1} ( z ) + \cR_{X_2} ( z )
\end{align}
where, for $ X \in A $, we let
$ R_X ( z ) = \sum_{ n =1 }^\infty \kappa_n ( X ) z^n  $
 and 
$ \cR_X ( z ) = \sum_{ n = 1 }^\infty  \ck_n ( X ) z^n $. That is, the mixed free and c-free cumulants in $ X_1 $ and $ X_2 $ vanish.

We will prove next a result analogous to Theorem 14.4 from \cite{ns}, more precisely a lemma about c-free cumulants with products as entries, in fact an operator-valued version of Lemma 3.2 from \cite{mvp-jcw}.

The Kreweras complementary
  $ \text{Kr}(\pi)$ of $ \pi \in  NC(n)$
  is defined as follows (see \cite{ns}, \cite{kreweras}). Consider the symbols
   $\overline{1},\dots,\overline{n}$
   such that
   $1<\overline{1}<2<\dots<n<\overline{n}$.
   Then $Kr(\pi)$ is the
biggest element of
 $NC(\overline{1},\dots,\overline{n})\cong NC(n)$
such that $ \pi\cup \text{Kr}(\pi) $ is an element of $ NC(1,\overline{1},\dots,n,\overline{n}) $.
The total number of blocks in $\gamma$ and $\kr(\gamma)$ is $n+1$.

For $ \gamma \in NC(n) $  and $ p \in \{ 1, 2, \dots, n \} $, we will denote by
$ \gamma [ p ] $ the block of $ \gamma $ that contains $ p $.

\begin{lem} \label{lemma:1}
Suppose that $ X, Y $ are two c-free elements of $ A $. Then
\begin{enumerate}
\item[(i)] $ \displaystyle
\kappa_n (XY) 
= 
\sum_{ \gamma\in NC(n) }
 \prod_{ B \in \gamma } \kappa_{ | B | } (X) 
\cdot \prod_{ D \in \text{Kr}(\gamma) } \kappa_{ | D | } ( Y ) $
\item[(ii)]
$ \displaystyle
\ck_n (XY) 
= 
\sum_{ \gamma\in NC(n) }
\ck_{ | \gamma [ 1 ] | } ( X ) \cdot
 \ck_{ | \text{Kr} ( \gamma) [ \overline{n} ] | } ( Y )
  \cdot
 \prod_{ \substack{ B \in \gamma\\ B \neq \gamma[ 1 ] } }
  \kappa_{ | B | } (X) 
\cdot \prod_{\substack{ D \in \text{Kr}(\gamma) \\ D \neq \text{Kr}( \gamma)[ \overline{n}] } } \kappa_{ | D | } ( Y ) $
\end{enumerate}

\end{lem}

\begin{proof}
Part (i) is shown in \cite{ns}, Theorem 14.4. We will show part (ii) by induction on $ n $. For $ n = 1 $, the statement is trivial, since $ \ck_2 ( X, Y ) = 0 $ from the c-freeness of $ X $  and $ Y $, therefore
$
 \ck_1 (XY) = \Phi ( X Y ) = \ck_1 (X) \ck_1 (Y ). 
$

For the inductive step, in order to simplify the writting,  we will introduce several new notations.

 Let $ NC_S ( n )=\{ \pi \in NC( n):$   elements from the same block of $ \pi $ have the same parity $\}$. For $\sigma\in NC_{S}( n)$, denote $\sigma_{+}$, respectively $\sigma_{-}$
the restriction of $\sigma$ to the even, respectively odd, numbers and define
 $
  NC_{0}(n)=\{\sigma:\sigma\in
NC(n),\sigma_{+}=Kr(\sigma_{-})\}.
 $
 
 Also, we will need to consider the mappings
   \[
NC(n)\times NC(m)\ni(\pi,\sigma)\mapsto\pi\oplus\sigma\in NC(m+n),
 \]
 the juxtaposition of partitions, and
  \[
NC(n)\ni\sigma\mapsto\widehat{\sigma}\in NC(2n)
 \]
 constructed by doubling the elements, that is if $ (i_1, i_2, \dots, i_s ) $
is a block of $ \sigma $, then 
$ (  2 i_1 -1, 2 i_1 , 2 i_2 -1, 2 i_2, \dots, 2 i_s - 1, 2 i_s ) $ 
is a block of  $ \widehat{\sigma} $.

Then, for $ \pi \in NC (n ) $, we define 
\begin{align*}
\kappa_\pi[ X_1, \dots, X_n ] 
& = 
\prod_{\substack{ C \in \pi \\ C = (i_1, \dots, i_l ) } }  \kappa_l (X_{ i_1}, \dots, X_{ i_l } )\\
\mK_\pi[ X_1, \dots, X_n ]
& = 
\prod_{\substack{ B \in \text{Ext} (  \gamma ) \\  B  = ( j_1, \dots, j_l ) } } \ck_l (  X_{ j_1 } \cdots X_{ j_l } ) ]
 \cdot 
[ 
\prod_{\substack{ D  \in \text{Int} ( \gamma ) \\  D = (i_1, \dots, i_s ) } }  \kappa_s (X_{ i_1}, \dots, X_{ i_s } )
\end{align*}

\noindent Remark that   $ \kappa_{ \pi \oplus \sigma } = \kappa_\pi \cdot \kappa_\sigma $ and 
 $ \mK_{ \pi \oplus \sigma } = \mK_\pi \cdot \mK_\sigma $. 
Also, if 
$ ( i_1, i_2, \dots, i_s )$ is an exterior block of $\pi$, 
then 
\[ \pi_{ | \{ i_1 + 1, \dots, i_s - 1 \} }  \cup \kr ( \pi_{ | \{ i_1 + 1, \dots, i_s - 1 \} } ) 
= \pi \cup \kr(\pi)_ { | \{ i_1 + 1, \dots, i_s - 1 \} },
\]
 so 
Lemma \ref{lemma:1} is equivalent to 
\begin{align*}
\kappa_{\pi}[XY,\dots,XY]& =\sum_{\substack{\sigma\in NC_{S}(2n)\\
\sigma\bigvee\widehat{0_{n}}=\widehat{\pi}}
}\kappa_{\sigma}[X,Y,\dots,X,Y]
\\
\mK_{\pi}[XY,\dots,XY]
 &=
 \sum_{\substack{\sigma\in NC_{S}(2n)\\
\sigma\bigvee\widehat{0_{n}}
 =\widehat{\pi}}
}\mK_{\sigma}[X,Y,\dots,X,Y],
\end{align*}
therefore\
$
\phi\left((XY)^{n}\right) = \ck_n (XY)+\sum_{\substack{\pi\in NC(n)\\
\pi\neq\mathbbm{1}_{n}} }\mK_{\pi}[XY,\dots,XY].
$

On the other hand,
 \begin{equation*}
\phi\left((XY)^{n}\right)=\phi(X\cdot Y\cdots X\cdot
Y)=\sum_{\sigma\in
NC(2n)}\mK_{\sigma}[X,Y,\dots,X,Y].
 \end{equation*}
 Since  the mixed cumulants vanish, the equation above becomes
  \begin{eqnarray*}
\phi\left((XY)^{n}\right) & = & \sum_{\sigma\in NC_{S}(2n)}\mK_{\sigma}[X,Y,\dots,X,Y]\\
 & = & \sum_{\sigma\in NC_{0}(2n)}\mK_{\sigma}[X,Y,\dots,X,Y]+\sum_{\substack{\sigma\in NC_{S}(2n)\\
\sigma\notin NC_{0}(2n)} }\mK_{\sigma}[X,Y,\dots,X,Y]
  \end{eqnarray*}

\noindent But
$NC_{S}(2n)=\bigcup_{\pi\in NC(n)}\{\sigma:\ \sigma\in
NC_{S}(2n),\sigma\bigvee\widehat{0_{n}}=\widehat{\pi}\}$,
  and, for
$\sigma\in NC_{s}(2n)$, one has that $\sigma\in NC_{0}(2n)$ if and
only if
 $\sigma\bigvee\widehat{0_{n}}=\mathbbm{1}_{2n}$,
  Therefore:
 \[
NC_{S}(2n)\setminus NC_{0}(2n)=\bigcup_{\substack{\pi\in NC(n)\\
\pi\neq\mathbbm{1}_{n}} }\{\sigma:\ \sigma\in
NC_{S}(2n),\sigma\bigvee\widehat{0_{n}}=\widehat{\pi}\},
 \]
henceforth
 \begin{eqnarray*}
\sum_{\substack{\sigma\in NC_{S}(2n)\\
\sigma\notin NC_{0}(2n)}
}\mK_{\sigma}[X,Y,\dots,X,Y]
 & = &
  \sum_{\substack{\pi\in NC(n)\\
\pi\neq\mathbbm{1}_{n}}
}\sum_{\substack{\sigma\in NC_{S}(2n)\\
\sigma\bigvee\widehat{0_{n}}=\widehat{\pi}}
}\mK_{\sigma}[X,Y,\dots,X,Y]\\
 & = &
  \sum_{\substack{\pi\in NC(n)\\
\pi\neq\mathbbm{1}_{n}} }\mK[XY,\dots,XY].
 \end{eqnarray*}
 so the proof is now complete.

\end{proof}

\subsection{Non-crossing linked partitions and $t$-coefficients} ${}$\\
 By a non-crossing linked partition $\pi$ of the ordered set
 $\{1,2,\dots,n\}$ we will understand a collection $B_1,\dots, B_k$
 of subsets of $\{1,2,\dots,n\}$, called blocks, with the following
 properties:
 \begin{enumerate}
 \item[(a)]$\displaystyle{\bigcup_{l=1}^kB_l=\{1,\dots,n\}}$
 \item[(b)]$B_1,\dots,B_k$ are non-crossing.
 \item[(c)] for any $1\leq l,s\leq k$, the intersection $B_l\bigcap
 B_s$ is either void or contains only one element. If
 $\{j\}=B_i\bigcap B_s$, then $|B_s|, |B_l|\geq 2$ and $j$ is the
 minimal element of only one of the blocks $B_l$ and $B_s$.
\end{enumerate}

 For $ \pi $ as above we define $s(\pi)$ to be the set of all $1\leq k\leq n$ such that there are no blocks of $\pi$ whose minimal element is $k$. A block $B=i_1<i_2<\dots <i_p$ of $\pi$ will be called \emph{exterior} if there is no other block
$D$ of $\pi$ containing two elements $l,s$ such that  $l\leq
i_1<i_p<s$. The set of all non-crossing linked partitions on $\{1,\dots, n\}$
 will be denoted by $NCL(n)$.

\textbf{Example 2}: Below is represented graphically the
non-crossing linked
partition
 $\pi=(1,4,6,9), (2,3), (4,5),(6,7,8),(10,11),
(11,12)$  from $ NCL(12)$. Its exterior blocks are $(1, 4, 6, 9)$ and $(10, 11)$.

 \setlength{\unitlength}{.13cm}
 \begin{equation*}
 \begin{picture}(10,8)

\put(-16,3){\circle*{1}}

\put(-12,3){\circle*{1}} \put(-12,3){\line(0,1){3}}

\put(-8,3){\circle*{1}}\put(-8,3){\line(0,1){3}}

\put(-4,3){\circle*{1}}

\put(0,3){\circle*{1}}\put(0,3){\line(0,1){3}}

\put(4,3){\circle*{1}}

\put(8,3){\circle*{1}}\put(8,3){\line(0,1){3}}

\put(12,3){\circle*{1}}\put(12,3){\line(0,1){3}}

\put(16,3){\circle*{1}}

\put(20,3){\circle*{1}}

\put(24,3){\circle*{1}}

\put(28,3){\circle*{1}}\put(28,3){\line(0,1){3}}

\put(-12,6){\line(1,0){4}}

\put(8,6){\line(1,0){4}}

\put(-4,3){\line(4,3){4}}

\put(4,3){\line(4,3){4}}

\put(24,3){\line(4,3){4}}


\put(-16.5,0.2){1}

\put(-12.5,0.2){2}

\put(-8.5,0.2){3}

\put(-4.5,0.2){4}

\put(-0.5,0.2){5}

\put(3.5,0.2){6}

\put(7.5,0.2){7}

\put(11.5,0.2){8}

\put(15.5,0.2){9}

\put(19,0.2){10}

\put(23,0.2){11}

\put(27,0.2){12}

\linethickness{.55mm}

\put(-16,3){\line(0,1){5}}

\put(-16,8){\line(1,0){32}}

\put(20,8){\line(1,0){4}}

\put(20,3){\line(0,1){5.1}}

\put(24,3){\line(0,1){5.1}}

\put(16,3){\line(0,1){5}}

\put(4,3){\line(0,1){5.1}}

\put(-4,3){\line(0,1){5}}

 \end{picture}
 \end{equation*}

 Similarly to \cite{dykema} and \cite{mp}, we define the $ t $-coefficients, respectively $ \ct $-coefficients, as follows. Take $ A^{\circ} = A \setminus  \ker \varphi $. Then, for  $ n $ a positive integer, the maps
  $ t_n : A \times ( A^\circ )^n \lra \mathbb{C} $ and 
$ {}^c t_n : A \times (A^\circ)^n \lra  B $  are given by the following recurrences:
\begin{equation}\label{eq:31}
\varphi( X_1\cdots X_n )
=
\sum_{ \pi\in \NCL ( n ) }  [  \prod_{\substack{  B \in \pi \\  B = ( i_1, \dots, i_l )  } }
 t_{ l-1 } ( X_{ i_1} \, \dots, X_{ i_l } ) 
\cdot \prod_{ p \in s( \pi ) } t_0 ( p ) ]\\
\end{equation}
respectively
\begin{eqnarray}\label{eq:32}
\Phi ( X_1 \cdots X_n ) 
&  = &
\sum_{ \pi \in \NCL(n) } [ \prod_{ \substack{ B \in \text{Ext}( \pi ) \\  B = ( i_1, \dots, i_l )  } } 
\ct_{ l-1 } ( X_{ i_1} \, \dots, X_{ i_l } ) \label{eq:41}\\
&&
\hspace{1cm}\cdot \prod_{ \substack{ D \in \text{Int} ( \pi ) \\ D = ( j_1, \dots , j_s ) } } 
t_{ s-1 } ( X_{ j_1} , \dots, X_{ j_s } ) 
\cdot 
\prod_{ p \in s(\pi ) } t_0 ( X_p ). \nonumber
\end{eqnarray}

To simplify the writing we will use the shorthand notations
 $ t_\pi[X_1,\dots, X_n] $,
respectively
$\ct_\pi[X_1,\dots, X_n]  $ 
for the summing term of the right-hand side of (\ref{eq:31}),  respectively (\ref{eq:41}); 
also, if $ X_1 = \dots = X_n= X$, we will use the shorter notations $ t_\pi[X] $,  $ \ct_\pi[X] $ respectively $t_n(X)$,  $ \ct_n(X)$.
Note that all the factors in $ t_\pi $ are $ t $-coefficients, but
 the development of $ \ct_\pi $ contains both $ \ct $- and $ t $-coefficients.

 \subsection{The lattice $NCL(n)$ and c-freeness in terms on $ t $-coefficients}

 On the set $NCL(n)$ we define a order
 relation by saying that $\pi\su \sigma$ if for any block $B$
 of $\pi$ there exist $D_1,\dots, D_s$ blocks of $\sigma$ such that
 $\displaystyle {B=D_1\cup\dots\cup D_s}$.
  With respect to the order relation $\su$, the set $NCL(n)$ is a lattice that contains $NC(n)$.

We say that $i$ and $j$ are \emph{connected }in $\pi\in NCL(n)$ if there
 exist $B_1,\dots,B_s$ blocks of $\pi$ such that $i\in B_1$, $j\in
 B_s$ and $B_k\cap B_{k+1}\neq \varnothing$, $1\leq k\leq s-1$.

 For every $\pi\in NCL(n)$ there exist a unique partition $c(\pi)\in
 NC(n)$ defined as follows: $i$ and $j$ are in the same block of
 $c(\pi)$ if and only if they are connected in $\pi$. 
  We
 will use the notation
 \[
 [c(\pi)]=\{\sigma\in NCL(n):
 c(\sigma)=c(\pi)\}.
 \]

 \noindent In  Example 2 from above, we have 
 $c(\pi)=(1, 4, 5, 6, 7, 8, 9), (2, 3),  (10, 11, 12)$ for
 $ \pi = (1, 4, 6, 9), (2, 3), (4, 5), (6, 7, 8), (10, 11), (11, 12) $.

 For every $\gamma \in NC(n)$, the set $[\gamma]$ is a sublattice of
  $NCL(n)$ with maximal element $\gamma$. Moreover,
 if $\gamma$ has the blocks $B_1,\dots, B_s$, each $B_l$ of cardinality
  $k_l$, then  we have the following ordered set
  isomorphism:

  \begin{equation}\label{factorization}
[c(\pi)] \simeq
[\mathbbm{1}_{k_1}]\times\cdots\times[\mathbbm{1}_{k_s}]
  \end{equation}

\begin{prop}\label{cum-tcoeff}
For any positive integer $n$ and any $X_1,\dots, X_n\in A $ we have that
\begin{align*}
\kappa_n(X_1,\dots,X_n)
&
=\sum_{\pi\in[\mathbbm{1}_n]}t_\pi[X_1,\dots, X_n]\\
\ck_n(X_1,\dots,X_n)
&
=\sum_{\pi\in[\mathbbm{1}_n]}\ct_\pi[X_1,\dots, X_n]
\end{align*}
\end{prop}

\begin{proof}
First relation is shown in Proposition 1.4 from \cite{mp}. 
The second relation is trivial for $ n = 1 $. For $ n > 1 $, note that 
\[
\sum_{\pi\in NCL(n)} \ct_\pi[X_1,\dots,X_n]
=\sum_{\gamma\in NC(n)}\sum_{\pi\in[\gamma]}
\ct_\gamma[X_1,\dots,X_n]
\]
and the similar relation for $ t_\pi $.

Suppose that $ \pi \in NCL(n) $ and $ \gamma \in NC(n) $ are such that $ \pi \in [ \gamma ] $.
From the definition of $ \ct $- and $ t $-coefficients we have that
\begin{equation}\label{eq:24}
\ct_\pi [ X_1, \cdots, X_n ]
=
\prod_{ \substack{  B \in \text{Ext} ( \pi )\\ B = ( i_1, \dots, i_s )  } }
\ct_{ \pi_{ | B } } [ X_{ i_1 }, \dots, X_{ i_s } ] 
\cdot 
\prod_{ \substack{  D \in \text{Int} ( \pi )\\ D = ( j_1, \dots, j_t )  } }
 t_{ \pi_{ | D } } [ X_{ j_1 }, \dots, X_{ j_t } ] 
\end{equation}
and 
 \begin{equation}\label{eq:25}
 t_\pi[X_1,\dots,X_n] = 
\prod_{ \substack{  B \in  \pi )\\ B = ( i_1, \dots, i_s )  } }
\ct_{ \pi_{ | B } } [ X_{ i_1 }, \dots, X_{ i_s } ] .
 \end{equation}

Therefore, the equation (\ref{eq:32}) becomes
\begin{align*}
\Phi( X_1 \cdots X_n ) 
= & \sum_{ \gamma \in NC(n) } 
[
 \prod_{ \substack { B \in \text{Ext}(\pi ) \\ B =  (  i_1, \dots, i_s ) } } 
(
\sum_{ \substack{ \pi \in NCL(n) \\ \pi \in [ \gamma ] } }
\ct_{ \pi_{ | B } } [ X_{ i_1 }, \dots, X_{ i_s } ] 
)
\\
 & 
\cdot
\prod_{ \substack { D \in \text{Int}(\pi ) \\ D =  (  j_1, \dots, j_l ) } } 
(
\sum_{ \substack{ \pi \in NCL(n) \\ \pi \in [ \gamma ] } }
t_{ \pi_{ | D } } [ X_{ j_1 }, \dots, X_{ j_l } ] 
)
]
\end{align*}
 and the factorization (\ref{factorization}) gives:
\begin{align*}
\Phi( X_1 \cdots X_n ) 
= & \sum_{ \gamma \in NC(n) } 
[
 \prod_{ \substack { B \in \text{Ext}(\pi ) \\ B =  (  i_1, \dots, i_s ) } } 
(
\sum_{\sigma\in[\mathbbm{1}_s] }
\ct_{ \sigma } [ X_{ i_1 }, \dots, X_{ i_s } ] 
)
\\
 & 
\cdot
\prod_{ \substack { D \in \text{Int}(\pi ) \\ D =  (  j_1, \dots, j_l ) } } 
(
\sum_{ \sigma\in[\mathbbm{1}_l]  }
t_{ \sigma } [ X_{ j_1 }, \dots, X_{ j_l } ] 
)
]
\end{align*}

The conclusion follows now utilizing the moment-cumulant recurrence and induction on $n$.

\end{proof}

For $ X \in A^\circ $, we define the formal power series
 $  T_X ( z ) = \sum_{ n = 0 }^\infty t_n ( X ) z^n $
 and
 $  {}^cT_X ( z ) = \sum_{ n = 0 }^\infty \ct_n ( X ) z^n $. Also,  we consider the moment series of $ X $, namely $ M_X(z)= \sum_{n=1}^\infty \Phi(X^n) z^n $, and $ m_X(z) = \sum_{ n =1}^\infty \varphi(X^n) z^n $. 
  As shown in \cite{mvp-jcw}, Lemma 7.1(i), the recurrence \ref{eq:31} gives that
  \begin{equation}\label{eq:38}
  T_X(m_X(z)) \cdot ( 1 + m_X(z)) =\frac{1}{z} m_X(z ).
  \end{equation}
The proposition below gives an analogous relation for the series $ \cT_X(z) $ (i. e. a non-commutative analogue of Lemma 7.1(ii) from \cite{mvp-jcw}).

\begin{prop}\label{analytic-ct}
 With the notations above, we have that
 \begin{equation}\label{eq:39}
 \cT_X (m_X(z) ) \cdot ( 1 + M_X( z ) ) = 
 \frac{1}{z} M_X(z).
 \end{equation}
\end{prop}

\begin{proof}

 Let $ NCL(1, p) = \{ \pi \in NCP(p) : \pi \  \text{has only one exterior block} \}$.
 
 Note that for each $ \tau \in NCL(n) $, there exists a unique triple $ p \leq n $, $ \pi \in NCL(1, p ) $ and $ \sigma \in NCL(n-p) $ such that
 $ \tau = \pi \oplus \sigma $. 
  Indeed, taking $ p $ to be the maximal element of the block of $ c(\tau) $ containing 1, it follows that 
   \[ \tau = \tau_{ | \{ 1, 2, \dots, p \}} \oplus
   \tau_{ | \{ p+1, \dots, n \}} \]
   and that $ \tau_{ | \{ 1, 2, \dots, p \}} \in NCL(1, p) $.
   
   Conversely, each triple $ p, \pi, \sigma $ as above determine a unique $ \pi \oplus \sigma \in NCL(n) $, hence
   \[ 
    NCL(n) = \bigcup_{ p \leq n } \{ \pi \oplus \sigma : \ \pi \in NCL(1, p), \sigma \in NCL( n - p ) \}
    \]
   therefore the recurrence \ref{eq:32} gives
   \begin{align}
   \Phi(X^n) &= \sum_{ \pi \in NCL(n)} \ct_pi(X) \nonumber\\
    & =
   \sum_{ p \leq n } [ \sum_{ \pi \in NCL(p)} \ct_\pi(X) \cdot ( \sum_{ \sigma \in NCL(n-p)} \ct_\sigma(X))] \nonumber \\
   & = \sum_{ p \leq n } [ \sum_{ \pi \in NCL(p)} \ct_\pi(X) \cdot \Phi(X^{ n - p}) ]\label{new:01}.
   \end{align}

   Let $ NCL (1, q, p ) = \{ \pi \in NCL(1, p) :  \
   \text{  the exterior block of $ \pi $ has exactly $ q $ elements} \} $.
   Fix $ \pi \in NCL(1, q, p) $ and let $ (1, i_1, i_2, \dots, i_{q-1} )$ be the exterior block of $ \pi $. Define $ \widetilde{\pi} \in NCL(p-1)$ such that, with the notations following the recurrences (\ref{eq:31})--(\ref{eq:32}),  
   $ \ct_\pi(X)= \ct_{ q-1}\cdot t_{ \widetilde{\pi}}(X) $,  as follows:
   \begin{enumerate}
   \item[-] if  $ (j_1, j_2, \dots, j_s ) $ is an interior block of $ \pi $, then $(j_1 -1, j_2-1, \dots, j_s-1 ) $ is a block in $ \widetilde{\pi} $;
   \item[-] if $ j >1 $ and  the only block of $ \pi $ containing $ j $ is exterior, then $ ( j-1) $ is a block of $ \widetilde{\pi} $
   \end{enumerate}
   i.e. $ \widetilde{\pi} $ is obtaining by `` deleting '' the 1 and the exterior block of $ \pi $. 
    For each $ i_l $ in the exterior block $(1, i_1, i_2, \dots, i_{q-1})$ of $ \pi $ we define $ i_l^\prime $ to be the maxinal element connected to $ i_l $ in $ \pi $, and let $ i_0^\prime =0$.
    Then each set
     $ S(l) = \{ i_{l-1}^\prime +1, i_{l-1}^\prime +2, \dots, i_l^\prime\} $ is nonvoid  and we have the decomposition
     $ \widetilde{\pi } = \widetilde{\pi }_{ | S(1)} \oplus \widetilde{\pi }_{ | S(2)}
     \oplus \dots \oplus \widetilde{\pi }_{ | S(q-1)} .$
   
   \smallskip
   
   \noindent \textbf{Example 3.} If
    $ \pi = \{(1, 4, 6, 9), (2, 3), (4, 5), (6, 7, 8)\} $,
     then  
   \[
    \widetilde{\pi}=\{ (1, 2), (3, 4), (5, 6, 7), (8) \} = \{ (1, 2), (3, 4)\} \oplus \{ (1, 2, 3)\} \oplus \{ (1)\},\]
    see the diagram below:
  
    \setlength{\unitlength}{.13cm}
    \begin{equation*}
    \begin{picture}(24,4)

         \put(20,0){\circle*{1}} 
         \put(20,0){\line(0,1){3}}

         \put(24,0){\circle*{1}}
         \put(24,0){\line(0,1){3}}
         
         \put(20,3){\line(1,0){4}}
         
         \put(28,0){\circle*{1}}
         \put(28,0){\line(0,1){3}}
         \put(28, 3){\line(1,0){4}}
         
         \put(32,0){\circle*{1}}
         \put(32,0){\line(0,1){3}}
         
         \put(33, 1){ $\oplus$}
         
         \put(38,0){\circle*{1}}
         \put(38,0){\line(0,1){3}}
         \put(38,3){\line(1,0){8}}
         
         \put(42,0){\circle*{1}}
         \put(42,0){\line(0,1){3}}
         
         \put(46,0){\circle*{1}}
         \put(46,0){\line(0,1){3}}
         
         \put(47,1){ $\oplus$}
         
         \put(52,0){\circle*{1}}
         
         \put(52,0){\line(0,1){3}}

         \put(4,3){\line(1,0){4}}
         
         \put(-8,0){\line(4,3){4}}
         
         \put(0,0){\line(4,3){4}}


   \put(-20,0){\circle*{1}}

   \put(-16,0){\circle*{1}}
    \put(-16,0){\line(0,1){3}}

   \put(-12,0){\circle*{1}}
   \put(-12,0){\line(0,1){3}}
   
   \put(-8,0){\circle*{1}}
   
   \put(-4,0){\circle*{1}}
   \put(-4,0){\line(0,1){3}}
   
   \put(0,0){\circle*{1}}
   
   \put(4,0){\circle*{1}}
   \put(4,0){\line(0,1){3}}
   
   \put(8,0){\circle*{1}}
   \put(8,0){\line(0,1){3}}
   
   \put(12,0){\circle*{1}}

   \put(-16,3){\line(1,0){4}}
   
   \put(4,3){\line(1,0){4}}
   
   \put(-8,0){\line(4,3){4}}
   
   \put(0,0){\line(4,3){4}}


   \linethickness{.55mm}
   
   \put(-20,0){\line(0,1){5}}
   
   \put(-20,5){\line(1,0){32}}

   \put(12,0){\line(0,1){5}}
   
   \put(0,0){\line(0,1){5}}
   
   \put(-8,0){\line(0,1){5}}
   
   
   \put(14,1 ){$\longrightarrow$}

   
      \end{picture}
       \end{equation*}
    
    \qquad
    
 Using the equality $ \ct_\pi(X)= \ct_{ q-1}\cdot t_{ \widetilde{\pi}}(X) $, we obtain
 \begin{align}
 \sum_{ \pi \in NCL(1, q, p)} \ct_\pi(X) 
 &= 
 \ct_{ q-1}(X)  \cdot 
 \sum_{ r_1 + \dots r_{ q-1}=p }
 ( \prod_{k=1}^{q-1} \sum_{ \sigma\in NCL(r_k)} t_{\sigma}(X) )\nonumber\\
 &=
 \ct_{ q-1}(X)\sum_{ r_1 + \dots r_{ q-1}=p }
 ( \prod_{k=1}^{q-1} \varphi(X^{r_k}). )\label{new:02}
 \end{align}
  
   Since $ \displaystyle  NCL(1, p) = \bigcup_{q=1}^p NCL(1, q, p) $, the equations (\ref{new:01}) and (\ref{new:02}) give that
  \[
  \Phi(X^n) = \sum_{p=1}^n ( \sum_{ q=1}^p \ct_{q-1}(X) \cdot\sum_{ r_1 + \dots r_{ q-1}=p }
   ( \prod_{k=1}^{q-1} \varphi(X^{r_k}),
  \]
  which is the relation of the left hand side and right hand side coefficients of $ z^{n-1} $ in equation  (\ref{eq:39}).
\end{proof}

We conclude this section with the following result.
\begin{prop}\label{vanishtcoeff}\emph{(Characterization of c-freeness in terms of ${}^ct$-coefficients)}\\
Two elements $ X,  Y $ from $ A^\circ $ are c-free if and only if all their mixed $ t $- and $ \ct $-coefficients vanish, that is
for all $  n $ and all $ a_1, \dots, a_n \in \{ X,  Y \} $ such that $ a_k = X $ and $ a_l = Y $ for some $ k , l $ 
we have that
\[
\ct_{ n -1}( a_1, \dots, a_n ) = t_{n-1} ( a_1, \dots, a_n ) = 0. 
\]
\end{prop}

\begin{proof}
We will show by induction on $ n $ the equivalence between vanishing of mixed free and c-free cumulants of order $ n $  in $ X $ and $ Y $ and vanishing of mixed $ t $- and $ \ct$-coefficients of order up to $ n -1 $.
For $ n = 2 $ the result is trivial, since $ k_2 ( X, Y ) = t_1 ( X, Y ) $ and $ \ck_2 ( X,  Y ) = \ct_1 ( X, Y ) $. 

For the inductive step suppose that $ a_1, \dots, a_n $ are not all $ X $ nor all $ Y $. Proposition \ref{cum-tcoeff}
gives
\begin{align*}
\kappa_n ( a_1, \dots, a_n )
 =&
t_{n-1}(a_1,\dots, a_n)+
\sum_{\substack{\pi\in[\mathbbm{1}_n]\\ \pi\neq \mathbbm{1}_n}}
t_\pi [a_1,\dots, a_n]
\\
\ck_n ( a_1, \dots, a_n ) =&
\ct_{n-1}(a_1,\dots, a_n)+
\sum_{\substack{\pi\in[\mathbbm{1}_n]\\ \pi\neq \mathbbm{1}_n}}
\ct_{ \pi } [ a_1, \dots, a_n ].
\end{align*}
Fix $ \pi \in [\mathbbm{1}_n] $, $ \pi \neq \mathbbm{ 1 }_n $. Since $ \pi $ is connected, there is 
$ ( i_1, \dots, i_s ) $, a block of $ \pi $ with $ s < n $ 
such that  $ a_{ i_1 }, \dots, a_{ i_2 } $ are not all $ X $ not all $ Y $, 
so equations (\ref{eq:24}) and (\ref{eq:25}) and the induction hypothesis  imply that $ t_\pi (a_1, \dots, a_n ) =\ct_{\pi }( a_1, \dots, a_n ) = 0 $, 
hence we have that 
$ \kappa_n ( a_1, \dots, a_n )
 =
t_{n-1}(a_1,\dots, a_n) $ 
and
$
\ck_n ( a_1, \dots, a_n ) =
\ct_{n-1}(a_1,\dots, a_n) $ so q.e.d..
\end{proof}

\section{planar trees and the multiplicative property of the $T$-transform}

In this section we will use the combinatorial arguments from \cite{mp} to show that whenever $ X  $ and $ Y $ are two c-free elements from $ A^\circ $, we have that 
 $ T_{ XY} (z ) = T_{ X } ( z ) \cdot T_{ Y } ( z ) $ and $ \T_{ XY} (z ) = \T_{ X} (z) \cdot \T_{ Y } (z ) $ 
 where, for $ Z \in A^\circ $, we define
$  T_Z ( z ) = \sum_{ n = 0 }^\infty t_n ( Z ) z^n $, respectively $  \T_Z ( z ) = \sum_{ n = 0 }^\infty {}^ct_n ( Z ) z^n $

  \subsection{Planar trees} ${}$
  
We will start with a review of the notations and results from \cite{mp}.

  By an \emph{elementary planar tree} we will denote a graph with
   $m\geq 1$ vertices,  $v_1,v_2,\dots, v_m$, and $m-1$ (possibly 0) edges, or branches,
    connecting the vertex $v_1$ (that we will call \emph{root}) to the  vertices $v_2, \dots, v_m$ (that we will call \emph{offsprings}).

  By a \emph{planar tree} we  will understand a graph consisting in a finite number
  of \emph{levels}, such that:
  \begin{enumerate}
  \item[-]first level consists in a single elementary planar tree,
  whose root will be considered the root of the planar tree;
   \item[-]the $k$-th level will consist in a set of elementary
    planar trees such that their roots are  offsprings from the $(k-1)$-th level.
      \end{enumerate}

 Below are represented graphically the elementary planar tree  $ C_1 $ and the 2-level planar tree $ C_2 $:

 \setlength{\unitlength}{.12cm}
 \begin{equation*}
 \begin{picture}(22,18)

\put(-22,8){\circle*{1}} \put(-22,8){\line(2,3){4}}

\put(-18,14){\circle*{1}}

\put(-18,8){\circle*{1}}\put(-18,8){\line(0,1){6}}

\put(-14,8){\circle*{1}} \put(-14,8){\line(-2,3){4}}

\put(-19,4){$ C_1$}

\put(8,16){\circle*{1}}

\put(5,11){\circle*{1}} \put(5,11){\line(3,5){3}}


\put(11,11){\circle*{1}} \put(11,11){\line(-3,5){3}}

\put(8,6){\circle*{1}} \put(8,6){\line(3,5){3}}

\put(14,6){\circle*{1}} \put(14,6){\line(-3,5){3}}

\put(5,2){$ C_2$}

\multiput(17,6)(4,0){3}{\line(1,0){2}}
\multiput(16,11)(4,0){3}{\line(1,0){2}}
\multiput(15,16)(4,0){3}{\line(1,0){2}}
\multiput(17,2)(4,0){3}{\line(1,0){2}}

\put(26,3){level 3}

\put(26,8){level 2}

\put(26,13){level 1}

 \end{picture}
 \end{equation*}
 
   The set of all planar trees with $ n $ vertices will be denoted by $  \mathfrak{T}(n) $.
  If $ C $ is a planar tree, the set of elementary trees composing it will be denoted by  $ E( C ) $, the elementary tree containing the root of $  C $
  will be denoted 
  by $ \mathfrak{r}(C) $, and we define 
 $ \mathfrak{b}(C) = E( C) \setminus \{ \mathfrak{r}(C) \} $.

  On the set of vertices of a planar tree we consider the ``\emph{left depth first}'' order from \cite{aep}, given by:
   \begin{enumerate}
   \item[(i)]roots are less than their offsprings;
   \item[(ii)] offsprings of the same root are ordered from left to right;
    \item[(iii)]if $v$ is less that $w$, then all the offsprings of $v$ are smaller than any offspring of $w$.
    \end{enumerate}
       ( I.e. the order in which the vertices are passed by walking along the branches from the root to the right-most vertex, not counting vertices passed more than one time, see the example below).

  \textbf{Example 4}:

   \setlength{\unitlength}{.12cm}
 \begin{equation*}
 \begin{picture}(18,18)

\put(8,18){1}

\put(8,16){\circle*{1.5}}

\put(3,11){\circle*{1.5}} \put(3,11){\line(1,1){5}}


\put(8,16){\line(2,-1){10}}

\put(18,11){\circle*{1.5}}

\put(18,11){\line(-3,-5){3}}

\put(18,11){\line(3,-5){3}}

\put(21,6){\circle*{1.5}}

\put(15,6){\circle*{1.5}}

\put(8,16){\line(0,-1){5}}

\put(8,11){\circle*{1.5}}

\put(7.5,8){4}

\put(3,6){\circle*{1.5}}\put(3,6){\line(0,1){5}}



\put(0,11){2}

\put(1.5,3 ){3}

\put(19.1,11.3){5}


\put(14,3.2){6}

\put(22,3.2){7}

 \end{picture}
 \end{equation*}

Next, consider, as in \cite{mp}, the map
 $\Theta: [\mathbbm{1}_n]\lra\mathfrak{T}(n)$
  such that $\Theta(\pi)$ is the unique planar tree with the property that if  $(i_1, \dots, i_s)$ is a block of $\pi$, then the vertices of $\Theta(\pi) $ form an elementary tree from $ E (\Theta(\pi))$.
   More precisely, if $(1,2,i_1,\dots, i_s)$ is the block of $\pi$ containing 1,
   then the first level of $\Theta(\pi)$ is the elementary planar tree of root
    numbered 1 and $ s + 1 $ offsprings numbered $(2,i_1,\dots, i_s)$.
     The second level of $\Theta(\pi)$ consists on the elementary trees representing the blocks
     (if any) having $2,i_1,\dots, i_s$ as first elements etc (see Example 5  below). 
As shown  in \cite{mp}, the map $ \Theta $ is well-defined and bijective.
     \medskip
     
     \noindent \textbf{Example 5}:
 \begin{center}
     \includegraphics[width=3in,height=1.1cm]{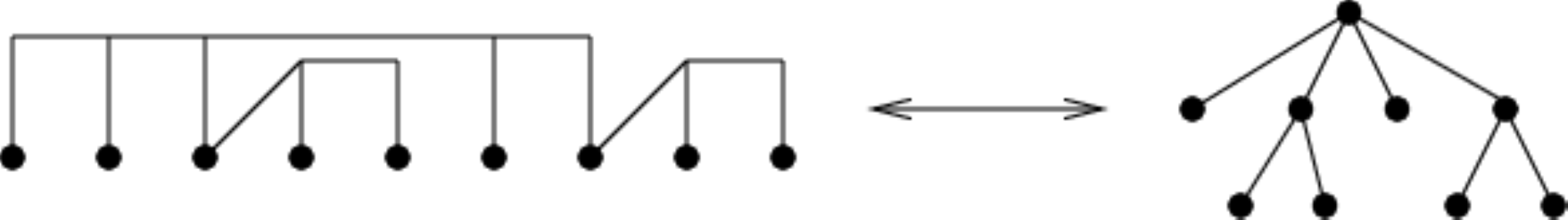}
\end{center}
    \medskip

For $ X \in A^\circ $, define the maps $ \cE_{ X} $, respectively $ \CE_{ X } $ from
 $  \bigcup_{ n \in \mathbb{N } } \mathfrak{T} ( n ) $ 
to 
$ \mathbb{C} $, respectively to $ B $, as follows. If $ C $ is an elementary planar tree with $ n$ vertices, then  
let $ \cE_X ( C ) = t_{ n-1} ( X ) $
and
$ \CE_X (C) = \ct_{ n - 1 } ( X ) $; for $ W  \in \mathfrak{T}(n ) $,  let
\begin{align*}
\cE_X( W ) 
& =
\prod_{ C \in E( W ) } \cE_X ( C ) \\
\CE_X ( W )
& = 
\CE_X ( \mathfrak{r} ( W ) ) \cdot \prod_{ C \in \mathfrak{b}(W) } \cE_X ( C ).
\end{align*} 

As in \cite{mp}, Proposition \ref{cum-tcoeff} and the bijectivity $\Theta $ give that
\begin{equation}\label{eq:10}
\kappa_n ( X ) = \sum_{ C \in \mathfrak{T}(n) } \cE_X ( C )   \ \  \text{and}  \   \
\ck_n ( X ) = \sum_{ C \in \mathfrak{T}(n) } \CE_X ( C ). 
\end{equation}

\subsection{Bicolor planar trees and the Kreweras complement}

  By a  \emph{bicolor elementary  planar tree} we will understand an elementary tree together with a mapping from its offsprings to $\{0, 1\}$ such that the offsprings whose image is 1 are smaller than (with respect to the order relation considered above, i.e. at right of)  the offsprings with image 0. Branches toward offsprings of color 0, respectively 1, will be also said to be of color 0, respectively 1. The set of all bicolor planar trees with $n$ vertices will be denoted by $\mathfrak{EB}(n)$.  Following \cite{mp},  branches of color 1 will be represented by solid lines and  branches of color 0 by dashed lines.  
  
  
   \noindent \textbf{Example 6}:
  The graphical representation of $\mathfrak{EB}(4)$:
  \smallskip
 \begin{center}
\includegraphics[width=3in,height=.8cm]{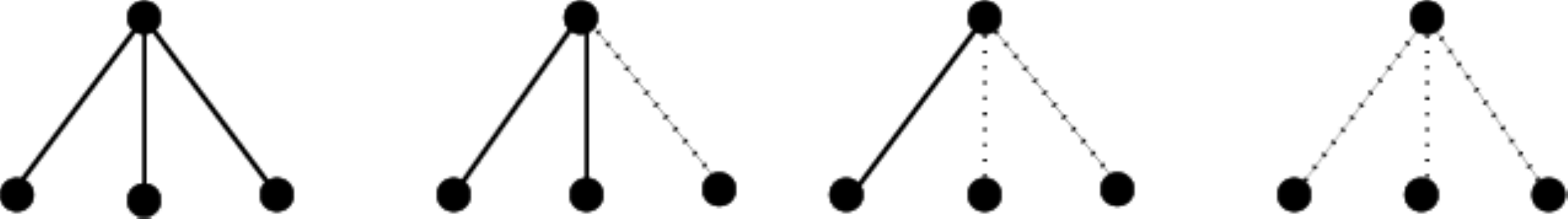}
\end{center}
\smallskip

  A   planar tree whose constituent elementary trees are all bicolor will be called \emph{bicolor planar tree} ; the set of all bicolor planar trees will be denoted by $\mathfrak{B}(n)$.

  Let $ NCL_S (2n) = \{ \pi\in NCL(2n) : $ elements from the same block of $ \pi $ have the same parity $\}$. We will say that the blocks of a partition from $ NCL_S(2n) $ with odd elements are of color 1 and the ones with even elements are of color 0;  blocks of color 1 will be graphically represented by solid lines and blocks of color 0 by dashed lines.

  \noindent \textbf{Example 7}: Representation of $ (1, 7), (2, 6), (3, 5), (4), (8, 12), (9, 11), (10) $:
  \medskip
\begin{center}
\includegraphics[width=3in,height=.9cm]{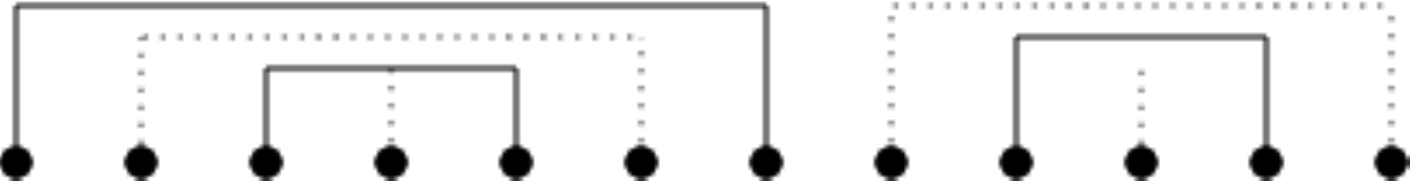}
\end{center}

\smallskip

As shown in \cite{mp}, there exist a bijection $ \Lambda : NCL_S ( 2n ) \lra \mathfrak{B}(n)  $, constructed as follows (see also Example 8 below):

   \begin{enumerate}
   \item[-] If $(i_1,\dots, i_s)$ and $(j_1,\dots, j_p )$ are the two exterior blocks of $\pi$, where $ j_1 =i_s + 1 $, then the first level of $\Lambda(\pi)$ is the elementary tree with $s-1+p-1$ offsprings,  the first $s-1$ of color 1, representing $(i_2,\dots, i_s)$, in this order, and the last $p-1$ of color 0, representing $(j_2,\dots, i_p)$, in this order.

   \item[-] Suppose that $i_1$ and $i_2$ are consecutive elements in a block $B$ of $\pi$ already represented in some tree of $\Lambda(\pi)$. Let $ t_2 = i_2 -1 $ and $ t_1 = j_s + 1 $ if $ i_1 $ is the minimal element of some block $ (j_1, j_2, \dots, j_s )$ of $ \pi $ other than $ B $ and $ t_1 = i_1 + 1 $ otherwise.
    Let $ B_1, \dots , B_s $ be, in this order, the exterior blocks with more than one element of $ \pi_{| t_1, t_1 + 1 \dots, t_2 } $ such that each $ B _k $ has $p(k) $ elements. 
    Finally, if $i_2 $ is the minimal elements of some block $ ( i_2, l_1, \dots, l_{p(s+1)} )  $ of $ \pi $, then let 
    $ D  =( i_2, l_1, \dots, l_{p(s+1)} ) $ (otherwise let $ p(s+1) = 0 $).
    Then we represent $ B_1, \dots, B_s, D $ by an elementary tree with root the vertex representing $ i_2 $ and of $ p(1) + \dots p(s+ 1) - s $ offsprings, where the first $ p(1) - 1 $ are representing, and keeping the color of, the vertices of $ B_1 $ except for the minimal one, the next $p(2) - 1 $ representing elements of $ B_2 $ except for the minimal one etc.
    
   \end{enumerate}

\newpage

\noindent \textbf{Example 8}:
  
\begin{center}

\includegraphics[width=8cm,height=2cm]{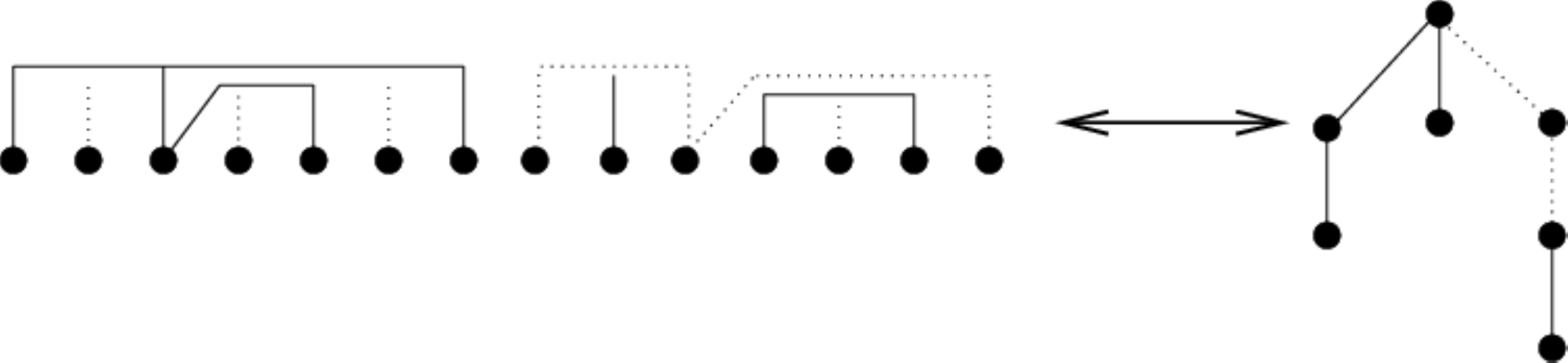}

\end{center}

\medskip

For $ X , Y \in A^\circ $, we define the maps 
$\displaystyle \omega_{XY}$ and $ \com_{XY} $ from $\cup_{ n\in \mathbb{N} } \mathfrak{B}( n ) $ to $\mathbb{C} $,
 respectively
 $  B$
 as follows.
 If
  $ C_0\in\mathfrak{EB}( n )$
   has $ k $ offsprings of color 1 and $ n - ( k +1 )  $ offsprings of color 0, then we define
\begin{align*}
\omega_{X,Y}(C_0)
 & =t_k(X)t_{n-k-1}(Y) 
\\
\com_{X,Y}(C_0)
 & =\ct_k(X) \ct_{ n - k -1 } ( Y ). 
\end{align*}
For  $  W \in\mathfrak{B} ( n ) $, define
\begin{align*}
\omega_{ X , Y }( W )
&
= \prod_{  D \in E ( W )  } \omega_{X,Y}( D )\\
\com_{ X,  Y } ( W )
&
=
\com_{ X, Y } ( \mathfrak{r} ( W ) ) 
\cdot \prod_{  D \in \mathfrak{b} ( W )  } \omega_{X,Y}( D ).
\end{align*} 
\noindent Remark that, for $ \pi \in NC_S ( 2 n ) $, the definitions of $ \Lambda $ and 
$ \omega_{ X, Y } $, $ \com_{ X ,  Y } $ give 
\begin{equation}
\label{eq:11}
\omega_{X,Y}(\Lambda(\pi))
=
\kappa_{\pi_-}[X]\kappa_{\pi_+}[Y],
\end{equation}
respectively 
\begin{equation}\label{eq:12}
\com_{X,Y}(\Lambda(\pi))
= 
\ck_{ | \pi_{-} [ 1 ] | } ( X ) 
\cdot
 \ck_{ | \pi_+ [ 2n ] | } ( Y )\cdot
   \prod_{ \substack{ B \in \pi_{ - }\\ B \neq \pi_{ - } [ 1 ] } } 
   \kappa_{ | B | } (X) 
  \cdot \prod_{\substack{ D \in \pi_{ + } \\ D \neq \pi_{ + } [ 2n] } }
 \kappa_{ | D | } ( Y )
\end{equation}

\subsection{The multiplicative property of the $ \cT$-transform}

\begin{thm} \label{mainthm}
If $ X,Y$ are c-free elements from $A^\circ$, then $T_{XY}(z)=T_X(z)T_Y(z)$
and $ \cT_{ XY}(z) =\cT_X (z)\cdot \cT_Y(z) $.
\end{thm}
\begin{proof}
 We need to show that, for all $m\geq 0$
 \begin{equation}\label{final}
 t_m(XY)=\sum_{k=0}^mt_k(X)t_{m-k}(Y))\  \ \text{and} \  
 \ct_m(XY)=\sum_{k=0}^m \ct_k(X) \ct_{m-k}(Y))
 \end{equation}

If $ C_n $ denotes the elementary planar tree with $ n $ vertices, with the notations from the previous two sections,
the equations (\ref{final}) are equivalent to
\begin{equation}\label{finaltree}
 \cE_{XY}(A_n)=\sum_{B\in\mathfrak{EB}(n)}\omega_{X,Y}(B) \  \ \text{and} \  
 \CE_{XY}(A_n)=\sum_{B\in\mathfrak{EB}(n)}\com_{X,Y}(B)
 \end{equation}

\noindent i.e. for example,
\setlength{\unitlength}{.15cm}
 \begin{equation*}
 \begin{picture}(28,7)

 \put(-21,4){$\cE_{XY}($}

\put(-14,3){\circle*{1}}

\put(-12,3){\circle*{1}}

\put(-10,3){\circle*{1}}

\put(-12,7){\circle*{1}}

\put(-14,3){\line(1,2){2}}

\put(-12,3){\line(0,1){4}}

\put(-10,3){\line(-1,2){2}}

\put(-9,4){$)$}

\put(-7,4){$=$}

\put(-5,4){$\omega_{XY}($}

\put(2,3){\circle*{1}}

\put(4,3){\circle*{1}}

\put(6,3){\circle*{1}}

\put(4,7){\circle*{1}}

\put(2,3){\line(1,2){2}}

\put(4,3){\line(0,1){4}}

\put(6,3){\line(-1,2){2}}

\put(7,4){$)$}

\put(9,4){$+$}

\put(11,4){$\omega_{XY}($}

\put(18,3){\circle*{1}}

\put(20,3){\circle*{1}}

\put(22,3){\circle*{1}}

\put(20,7){\circle*{1}}

\put(18,3){\line(1,2){2}}

\put(20,3){\line(0,1){4}}

\multiput(20,7)(.25,-.5){8}{\line(0,1){.2}}

\put(23,4){$)$}

\put(25,4){$+$}

\put(27,4){$\omega_{XY}($}

\put(34,3){\circle*{1}}

\put(36,3){\circle*{1}}

\put(38,3){\circle*{1}}

\put(36,7){\circle*{1}}

\put(34,3){\line(1,2){2}}

\multiput(36,3)(0,.5){8}{\line(0,1){.2}}

\multiput(36,7)(.25,-.5){8}{\line(0,1){.2}}

\put(39,4){$)$}

\put(41,4){$+$}

\put(43,4){$\omega_{XY}($}

\put(50,3){\circle*{1}}

\put(52,3){\circle*{1}}

\put(54,3){\circle*{1}}

\put(52,7){\circle*{1}}

\multiput(52,7)(-.25,-.5){8}{\line(0,1){.2}}

\multiput(52,3)(0,.5){8}{\line(0,1){.2}}

\multiput(52,7)(.25,-.5){8}{\line(0,1){.2}}

\put(55,4){$)$}

 \end{picture}
 \end{equation*}

We will prove (\ref{finaltree}) by induction on $ n $. For $  n =0$, the result is trivial. Suppose (\ref{finaltree}) true for $m\leq n-1$.

Relation (\ref{eq:10}) and Lemma \ref{lemma:1} give
\begin{align*}
\sum_{C \in\mathfrak{T}(n)}\CE_{XY}( C )
&=
\ck_n( X Y )\\
=&
\sum_{ \pi \in NC_S( 2n ) }
\ck_{ | \pi_{-} [ 1 ] | } ( X ) 
\cdot
 \ck_{ | \pi_+ [ 2n ] | } ( Y ) 
  \cdot
  \prod_{ \substack{ B \in \pi_{ - }\\ B \neq \pi_{ - } [ 1 ] } } 
  \kappa_{ | B | } (X) 
\cdot \prod_{\substack{ D \in \pi_{ + } \\ D \neq \pi_{ + } [ 2n] } }
 \kappa_{ | D | } ( Y )
\end{align*}
 \noindent and equation (\ref{eq:12}) and the bijectivity of $ \Lambda $ give:
 \begin{equation}\label{eq1:7}
 \sum_{C \in\mathfrak{T}(n)}\CE_{XY}( C )
  =\sum_{B\in\mathfrak{B}(n)}\com_{X,Y}(B).
  \end{equation}
  
  Similarly, we have that
  \begin{equation}\label{eq1:8}
  \sum_{C \in\mathfrak{T}(n)}\cE_{XY}( C )
    =\sum_{B\in\mathfrak{B}(n)}\omega_{X,Y}(B).
  \end{equation}

 All non-elementary trees from $\mathfrak{T}(n)$ are composed of elementary trees with less than $ n $ vertices. The relations (\ref{eq1:7})
 and (\ref{eq1:8}) imply that the image under $\CE_{XY}$, respectively
 $ \cE_{ XY } $ of any such tree is the sum of the images under
 $ \com_{ XY } $, respectively under $\omega_{XY}$ of its colored versions. Hence
 \begin{equation*}
 \sum_{\substack{C\in\mathfrak{T}(n)\\C\neq C_n}}\cE_{XY}(C)
  =
   \sum_{\substack{B\in\mathfrak{B}(n)\\B\notin\mathfrak{EB}(n)}}
   \omega_{X,Y}(B)
   \  \  \
   \text{and} \
   \sum_{\substack{C\in\mathfrak{T}(n)\\C\neq C_n}}\CE_{XY}(C)
     =
      \sum_{\substack{B\in\mathfrak{B}(n)\\B\notin\mathfrak{EB}(n)}}
      \com_{X,Y}(B).
 \end{equation*}

  \noindent Finally,  (\ref{eq1:7}), (\ref{eq1:8}) and the  two equations above give that 
  \begin{equation*}
  \cE_{XY}(A_n)=\sum_{B\in\mathfrak{B}(n)}\omega_{X,Y}(B) \ \ \ \text{and} \  \ 
  \CE_{XY}(A_n)=\sum_{B\in\mathfrak{B}(n)}\com_{X,Y}(B)
  \end{equation*}
which imply (\ref{finaltree}).
\end{proof}

\section{Infinite Divisibility} Fix a unital $C^{*}$-subalgebra $B$ of $L(H)$, the $C^{*}$-algebra of bounded linear operators on a Hilbert space $H$. In this section, we study the infinite divisibility relative to the c-freeness. A natural framework for such a discussion is in a \emph{c-free probability space} $(A,\Phi,\varphi)$, that is, the algebras $A$ is also a concrete $C^{*}$-algebra acting on some Hilbert space $K$, the linear map $\Phi:A\rightarrow B$ is a unital completely positive map, and the expectation functional $\varphi$ is a state on $L(K)$. Note that we have the norm $||\Phi||=||\Phi(1)||=1$. 

The \emph{distribution} of a unitary $u \in (A,\Phi,\varphi)$, written as the spectral integral \[u=\int_{\mathbb{T}}\xi\,dE_u (\xi),\] is the pair $(\mu, \nu)$, where $\nu=\varphi \circ E_{u}$ is a positive Borel probability measure on the circle $\mathbb{T}=\{|\xi|=1\}$, and $\mu$ is a linear map from $\mathbb{C}[\xi, 1/\xi]$, the ring of Laurent polynomials, into the $C^{*}$-algebra $B$ such that 
\[\mu(f)=\Phi(f(u, u^{*})), \quad f\in \mathbb{C}[\xi,1/\xi].\]
 Of course, the positivity of $\Phi$ and the Stinespring theorem (see \cite{paulsen}, Theorem 3.11) imply that the map $\mu$ extends to a completely positive map on $C(\mathbb{T})$, the $C^{*}$-algebra of continuous functions on $\mathbb{T}$. More generally, given any sequence $\{A_n\}_{ n\in\mathbb{Z}} \subset L(H)$, it is known (see \cite{paulsen}) that the \emph{operator-valued trigonometric moment sequence}
 \[A_n=\mu(\xi^n),\quad n\in \mathbb{Z},\]
  extends linearly to a completely positive map $\mu:C(\mathbb{T})\rightarrow L(H)$ if and only if the operator-valued power series  $F(z)=A_0/2+\sum_{k=1}^{\infty}z^kA_k$ converges on the open unit disk $\mathbb{D}$ and satisfies $F(z)+F(z)^{*}\geq 0$ for $z\in\mathbb{D}$. In particular, for the unitary $u$ this implies that its moment generating series \[M_{u}(z)=\Phi(zu(1-zu)^{-1})=\sum_{k=1}^{\infty}z^k\mu(\xi^k)\] and \begin{equation}\label{eq:18} m_{u}(z)=\varphi(zu(1-zu)^{-1})=\int_{\mathbb{T}}\frac{z\xi}{1-z\xi}\,d\nu(\xi)
  \end{equation} 
  satisfy the properties:
   \[I+M_{u}(z)+M_{u}(z)^{*} \geq 0\quad \text{and}\quad 1+m_{u}(z)+\overline{m_{u}(z)} \geq 0\] 
   for $|z|<1$. Thus, the formula
   \begin{equation}\label{B-trans}
   B_{u}(z)=\frac{1}{z}M_{u}(z)(I+M_{u}(z))^{-1},\quad z\in \mathbb{D},
   \end{equation}
   defines an analytic function from the disk $\mathbb{D}$ to the algebra $B$, with the norm $||B_{u}(z)||<1$ for $z\in\mathbb{D}$. Analogously, the function \[b_{u}(z)=\frac{m_{u}(z)}{z+zm_{u}(z)},\quad z\in\mathbb{D},\] will be an analytic self-map of the disk $\mathbb{D}$. Notice that we have $B_{u}(0)=\Phi(u)$ and $b_{u}(0)=\varphi(u)$.

Conversely, suppose we are given two analytic maps $B:\mathbb{D}\rightarrow B$ and $b:\mathbb{D}\rightarrow \mathbb{D}$ satisfying $||B(z)||<1$ for $|z|<1$. Then the maps $M(z)=zB(z)(I-zB(z))^{-1}$ and $m(z)=zb(z)/(1-zb(z))$ are well-defined in $\mathbb{D}$, and they can be written as the convergent power series: \[M(z)=\sum_{k=1}^{\infty}z^kA_k\quad \text{and} \quad m(z)=\sum_{k=1}^{\infty}a_kz^k,\quad z\in \mathbb{D},\] where the operators $A_k\in B$ and the coefficients $a_k \in \mathbb{D}$. Since $I-|z|^2B(z)B(z)^{*} \geq 0$, we have \[I+M_{u}(z)+M_{u}(z)^{*}=(I-zB(z))^{-1}(I-|z|^2B(z)B(z)^{*})[(I-zB(z))^{-1}]^{*} \geq 0\] for every $z \in \mathbb{D}$. Therefore, the solution of the operator-valued trigonometric moment sequence problem implies that the map \[
\mu(\xi^{n})=\begin{cases}
A_{n}, & n>0;\\
I, & n=0;\\
A_{n}^{*}, & n<0.
\end{cases}
\] extends linearly to a completely positive map from $C(\mathbb{T})$ into $B$. In the case of $m(z)$, we obtain a Borel probability measure $\nu$ on $\mathbb{T}$ satisfying \eqref{eq:18}. The pair $(\mu,\nu)$ is uniquely determined by the analytic maps $B$ and $b$. It is now easy to construct a c-free probability space $(A,\Phi,\varphi)$ and a unitary random variable $u\in A$ so that the distribution of $u$ is precisely the pair $(\mu,\nu)$. Indeed, we simply let $A=C(\mathbb{T})$, whose members are viewed as the multiplication operators acting on the Hilbert space $L^{2}\left(\mathbb{T};\nu\right)$, $\Phi=\mu$, and the variable $u$ can be defined as \[(uf)(\xi)=\xi f(\xi),\quad \xi \in \mathbb{T},\quad f\in L^{2}\left(\mathbb{T};\nu\right).\] In summary, we have identified the distribution of $u$ with the pair $(B_u,b_u)$ of contractive analytic functions. 

A unitary $u\in A$ is said to be \emph{c-free infinitely divisible} if for every positive integer $n$, there exists identically distributed c-free unitaries
$u_1,u_2,\cdots,u_n$ in $A$ such that $u$ and the product $u_1u_2\cdots u_n$ have the same distribution.

It follows from the definition of the c-freeness that if a unitary $u \in A$, with the distribution $(\mu,\nu)$, is c-free infinitely divisible, then the law $\nu$ must be infinitely divisible with respective to the \emph{free multiplicative convolution} $\boxtimes$, that is, to each $n \geq 1$ there exists a probability measure $\nu_n$ on $\mathbb{T}$ such that \[\nu=\nu_n \boxtimes \nu_n \boxtimes \cdots \boxtimes \nu_n \quad (n\,\,\text{times}).\] 

The theory of $\boxtimes$-infinite divisibility is well-understood, see [5], and we shall focus on the c-free infinitely divisible distribution $\mu$, or equivalently, on the function $B_u$.
 From Equation 
 (\ref{B-trans}) and Proposition \ref{analytic-ct}, we have that the $\cT$-transform of $u$ satisfies \[
 \cT_{u}\left(m_{u}(z)\right)=B_{u}(z).\quad 
 \] 
 Therefore, Theorem 2.1 yields immediately the following characterization of c-freely infinite divisibility.

\begin{prop} A unitary $u \in (A,\Phi,\varphi)$ with distribution $(\mu,\nu)$ is c-free infinitely divisible if and only if $\nu$ is $\boxtimes$-infinitely divisible and the function $B_u$ is infinitely divisible in the sense that to each $n \geq 1$, there exists an analytic map $B_n:\mathbb{D}\rightarrow B$ such that $||B_n(z)||<1\quad \text{and} \quad B_u(z)=\left[B_n(z)\right]^n,\quad z\in \mathbb{D}.$ 
\end{prop} 

It was proved in [5] that a $\boxtimes$-infinitely divisible law $\nu$ is the Haar measure $d\theta/2\pi$ on the circle group $\mathbb{T}=\{\exp(i\theta):\theta\in (-\pi,\pi]\}$ if and only if $\nu$ has zero first moment.  We now show the c-free analogue of this result. 
\begin{prop} Let $u \in (A,\Phi,\varphi)$ be a c-free infinitely divisible unitary with $\Phi(u)=0$. If $\varphi(u)=0$, then one has $\Phi(u^n)=0$ for all integers $n \neq 0$.\end{prop}
\begin{proof} Denote by $(\mu,\nu)$ the distribution of $u$. Assume first that $\varphi(u)=0$, hence the law $\nu$  equals $d\theta/2\pi$. The c-free infinitely divisibility of $u$ shows that there exist c-free and identically distributed unitaries $u_1$ and $u_2$ in $A$ such that $\varphi(u_1)=0=\varphi(u_2)$ and $u=u_1u_2$ in distribution. Therefore, for $n>1$, we have \begin{eqnarray*}
\Phi(u^n) 
 & = &
  \Phi(
\underbrace{(u_1u_2)(u_1u_2)\cdots(u_1u_2)}_{n\:\text{times}})\\
 & = &
 \Phi(u_1)\Phi(u_2)\cdots\Phi(u_1)\Phi(u_2)\\&=& \Phi(u_1u_2)\Phi(u_1u_2)\cdots\Phi(u_1u_2)=\Phi(u)^n=0.
 \end{eqnarray*} The case of $n<0$ follows from the identity $\Phi(u^{n})=\Phi(u^{-n})^{*}$.\end{proof}

An interesting case for c-freely infinite divisibility arises from the commutative situation. To illustrate, suppose $B=C(X)$, the algebra of continuous complex-valued functions defined on a Hausdorff compact set $X\subset\mathbb{C}$ equipped with the usual supremum norm. Denote by $\mathcal{M}$ the family of all Borel finite (positive) measures on $\mathbb{T}$, equipped with the weak*-topology from duality with continuous functions on $\mathbb{T}$. Then we shall have the following \begin{prop} Let $\nu$ be a $\boxtimes$-infinitely divisible law on $\mathbb{T}$ and $\nu \neq d\theta/2\pi$. A unitary $u \in (A,\Phi,\varphi)$ is c-free infinitely divisible if and only if its $\cT$-transform admits the following L\'{e}vy-Hin\v{c}in type representation: \[\cT_{u}\left(\frac{z}{1-z}\right)(x)=\gamma_{x} \exp \left(\int_{\xi \in \mathbb{T}}\frac{\xi z +1}{\xi z -1}\,d\sigma_{x}(\xi)\right), \quad x\in X, \quad z \in \mathbb{D},\] where the map $x \mapsto \gamma_{x}$ is a continuous function from $X$ to the circle $\mathbb{T}$ and the map $x \mapsto \sigma_{x}$ is weak*-continuous from $X$ to $\mathcal{M}$. 
\end{prop} 
\begin{proof} The integral representation follows directly from the characterization of c-free infinite divisibility in the scalar-valued case \cite{mvp-jcw}. To conclude, we need to show the continuity of the functions $\sigma_{x}$ and $\gamma_x$. To this purpose, observe that \[\left|\cT_{u}\left(\frac{z}{1-z}\right)(x)\right|=\exp \left(-\int_{\xi \in \mathbb{T}}\frac{1-|z|^2}{|z - \xi |^2}\,d\sigma_{x}(1/\xi)\right).\] In particular, we have \[\exp \left(-\sigma_{x}(\mathbb{T})\right)=|\cT_{u}(0)(x)|.\] Thus, if $\{x_\alpha\}$ is a net converging to a point $x \in X$, then the family $\{\sigma_{x_\alpha}(\mathbb{T})\}$ is bounded, and for every $z \in \mathbb{D}$ we have \[\int_{\xi \in \mathbb{T}}\frac{1-|z|^2}{|z - \xi |^2}\,d\sigma_{x_\alpha}(1/\xi)\rightarrow \int_{\xi \in \mathbb{T}}\frac{1-|z|^2}{|z - \xi |^2}\,d\sigma_{x}(1/\xi)\] as $x_\alpha \rightarrow x$. Since bounded sets in $\mathcal{M}$ are compact in the weak-star topology, the above convergence of Poisson integrals determines uniquely the limit $\sigma_x$ of the net $\{\sigma_{x_\alpha}\}$. We deduce that the measures $\{\sigma_{x_\alpha}\}$ converge weakly to $\sigma_{x}$, and therefore the function $\sigma_x$ is continuous. This and the identity \[\cT_{u}\left(0\right)(x)=\gamma_{x} \exp\left(-\sigma_{x}(\mathbb{T})\right)\] imply further the continuity of the function $\gamma_x$. 
\end{proof}

We end the paper with the remark that if the algebra $B$ is non-commutative, a $B$-valued contractive map $B(z)$ does not always have contractive analytic $n$-th roots for each $n\geq 1$, even when $B(z)$ is invertible for every $z\in \mathbb{D}$. To illustrate, consider the constant map \[
B(z)=\left[\begin{array}{cc}
\lambda^{2} & 2\lambda\\
0 & \lambda^{2}
\end{array}\right],\quad z \in \mathbb{D}.
\] Then $B(z)$ is an invertible contraction if the modulus $|\lambda|$ is sufficiently small, and yet its square roots \[
\left[\begin{array}{cc}
\pm \lambda & 1\\
0 & \pm \lambda
\end{array}\right]
\] have the same norm $\sqrt{1+|\lambda|^2}$.


\subsection*{Acknowledgements}
This work was partially supported by a grant of the Romanian National Authority for Scientific Research, CNCS – UEFISCDI, project number PN-II-ID-PCE-2011-3-0119.

\end{document}